   \pdfoutput=1

\documentclass{article}
\usepackage{excludeonly}

\usepackage{savesym}
\usepackage{scrextend}  
\usepackage{amsthm,amsmath}
\usepackage{amssymb,latexsym,graphicx}
\usepackage{amscd}
 \usepackage[all,cmtip]{xy}
\usepackage{tikz-cd}
  \usetikzlibrary{arrows}
\tikzset{
  commutative diagrams/.cd,
  arrow style=tikz
  }
\usepackage{tikz}

\usepackage{accents}
\usepackage{cite}
\usepackage{mathtools}
\usepackage{stmaryrd} 

\usepackage{calligra}
\DeclareMathAlphabet{\mathcalligra}{T1}{calligra}{m}{n}
\DeclareMathAlphabet{\mathpzc}{OT1}{pzc}{m}{it}
\usepackage{hycolor}
\usepackage{xcolor}
\usepackage[
                       colorlinks=true,
                       linkcolor=black, 
                       citecolor=black, 
                       urlcolor=blue,
                     ]{hyperref}
\usepackage{soul} 

\usepackage{makeidx}

\usepackage[intoc]{nomentbl}
\makenomenclature

\usepackage{stackengine} 
\usepackage{booktabs} 

\newtheorem{theorem}{Theorem}[section]

\newtheorem{corollary}[theorem]{Corollary}

\newtheorem{lemma}[theorem]{Lemma}
\newtheorem{proposition}[theorem]{Proposition}

\theoremstyle{definition}
\newtheorem{definition}[theorem]{Definition}

\newtheorem{remark}[theorem]{Remark}

\theoremstyle{remark}

\newcommand{\A}{{\mathbb{A}}}

\newcommand{\N}{{\mathbb{N}}}

\newcommand{\R}{{\mathbb{R}}}
\renewcommand{\SS}{{\mathbb{S}}}

\newcommand{\Z}{{\mathbb{Z}}}
\newcommand{\Aa}{{\mathcal{A}}}   

\newcommand{\Ll}{{\mathcal{L}}}   

\newcommand{\Ss}{{\mathcal{S}}}

\newcommand{\Vv}{{\mathcal{V}}}

\newcommand{\Yy}{{\mathcal{Y}}}

\newcommand{\AAA}{\mbf{A}}       
\newcommand{\AAAdot}{\mbf{\dot A}} 
\newcommand{\BBB}{\mbf{B}}       
\newcommand{\rot}{\mathrm{rot}}  
\newcommand{\ppp}{\mbf{p}}       %
\newcommand{\id}{{\rm id}}                

\newcommand{\cgraph}[1]{\Gamma_{\kern-.5ex{}#1}}     

\newcommand{\lambdacan}{\lambda_{\rm can}} 
\newcommand{\omegacan}{\omega_{\rm can}}   

\newcommand{\Arnold}{{Arnol$'$d}}           
\newcommand{\h}{{\mathfrak h}}     
\newcommand{\Qfrak}{{\mathfrak Q}}

\newcommand{\inner}[2]{\langle #1, #2\rangle}   
\newcommand{\INNER}[2]{\left\langle #1, #2\right\rangle}

\newcommand{\mbf}[1]{\text{\boldmath $#1$}}  

\def\Nablatop#1{\nabla^{#1}\kern-.5ex{}}

\def\NABLA#1{{\mathop{\nabla\kern-.5ex\lower1ex\hbox{$#1$}}}}
\def\Nabla#1{\nabla\kern-.5ex{}_{#1}}
\def\Tabla#1{\Tilde\nabla\kern-.5ex{}_{#1}}
\def\Babla#1{\widebar\nabla\kern-.5ex{}_{#1}}
\def\abs#1{\mathopen|#1\mathclose|}   
\def\Abs#1{\left|#1\right|}

\renewcommand{\Tilde}{\widetilde}

\newcommand{\p}{{\partial}}

\renewcommand{\1}{{{\mathchoice {\rm 1\mskip-4mu l} {\rm 1\mskip-4mu l}
{\rm 1\mskip-4.5mu l} {\rm 1\mskip-5mu l}}}}

\newlength\eqshift
\renewcommand\theequation{\thesection.\arabic{equation}}
\let\savetheequation\theequation

\makeatletter
\renewcommand*\env@matrix[1][\arraystretch]{%
  \edef\arraystretch{#1}%
  \hskip -\arraycolsep
  \let\@ifnextchar\new@ifnextchar
  \array{*\c@MaxMatrixCols c}}
\makeatother

\makeatletter
\let\save@mathaccent\mathaccent
\newcommand*\if@single[3]{%
  \setbox0\hbox{${\mathaccent"0362{#1}}^H$}%
  \setbox2\hbox{${\mathaccent"0362{\kern0pt#1}}^H$}%
  \ifdim\ht0=\ht2 #3\else #2\fi
  }
\newcommand*\rel@kern[1]{\kern#1\dimexpr\macc@kerna}
\newcommand*\widebar[1]{\@ifnextchar^{{\wide@bar{#1}{0}}}{\wide@bar{#1}{1}}}
\newcommand*\wide@bar[2]{\if@single{#1}{\wide@bar@{#1}{#2}{1}}{\wide@bar@{#1}{#2}{2}}}
\newcommand*\wide@bar@[3]{%
  \begingroup
  \def\mathaccent##1##2{%
    \let\mathaccent\save@mathaccent
    \if#32 \let\macc@nucleus\first@char \fi
    \setbox\z@\hbox{$\macc@style{\macc@nucleus}_{}$}%
    \setbox\tw@\hbox{$\macc@style{\macc@nucleus}{}_{}$}%
    \dimen@\wd\tw@
    \advance\dimen@-\wd\z@
    \divide\dimen@ 3
    \@tempdima\wd\tw@
    \advance\@tempdima-\scriptspace
    \divide\@tempdima 10
    \advance\dimen@-\@tempdima
    \ifdim\dimen@>\z@ \dimen@0pt\fi
    \rel@kern{0.6}\kern-\dimen@
    \if#31
      \overline{\rel@kern{-0.6}\kern\dimen@\macc@nucleus\rel@kern{0.4}\kern\dimen@}%
      \advance\dimen@0.4\dimexpr\macc@kerna
      \let\final@kern#2%
      \ifdim\dimen@<\z@ \let\final@kern1\fi
      \if\final@kern1 \kern-\dimen@\fi
    \else
      \overline{\rel@kern{-0.6}\kern\dimen@#1}%
    \fi
  }%
  \macc@depth\@ne
  \let\math@bgroup\@empty \let\math@egroup\macc@set@skewchar
  \mathsurround\z@ \frozen@everymath{\mathgroup\macc@group\relax}%
  \macc@set@skewchar\relax
  \let\mathaccentV\macc@nested@a
  \if#31
    \macc@nested@a\relax111{#1}%
  \else
    \def\gobble@till@marker##1\endmarker{}%
    \futurelet\first@char\gobble@till@marker#1\endmarker
    \ifcat\noexpand\first@char A\else
      \def\first@char{}%
    \fi
    \macc@nested@a\relax111{\first@char}%
  \fi
  \endgroup
}
\makeatother

\newcommand{\Jbar}{{\widebar{J}}}     

\def\XXint#1#2#3{{\setbox0=\hbox{$#1{#2#3}{\int}$}
     \vcenter{\hbox{$#2#3$}}\kern-.5\wd0}}

\long\def\symbolfootnote[#1]#2{\begingroup%
\def\thefootnote{\fnsymbol{footnote}}\footnote[#1]{#2}\endgroup}

\tikzset{
  symbol/.style={
    draw=none,
    every to/.append style={
      edge node={node [sloped, allow upside down, auto=false]{$#1$}}}
  }
}

\begin{document}
\sloppy

\author{\quad Urs Frauenfelder \quad \qquad\qquad
             Joa Weber\footnote{
  Email: urs.frauenfelder@math.uni-augsburg.de
  \hfill
  joa@unicamp.br
  }
        %
        %
    \\
    Universit\"at Augsburg \qquad\qquad
    UNICAMP
}

\title{Merry-go-round \\ and \\ time-dependent symplectic forms}


\date{\today}

\maketitle 
%


%
%

%





\begin{abstract}
In the merry-go-round fictitious forces are acting
like centrifugal force and Coriolis force.
Like the Lorentz force Coriolis force is velocity dependent
and, following {\Arnold}~\cite{Arnold:1961a},
can be modeled by twisting the symplectic form.
If the merry-go-round is accelerated an additional fictitious force
shows up, the Euler force.

In this article we explain how one deals symplectically with
the Euler force by considering time-dependent symplectic forms.
It will turn out that to treat the Euler force one also needs
time-dependent primitives of the time-dependent symplectic forms.
\newline
Time-dependent symplectic forms are motivated by the
elliptic restricted three-body-problem and its relation to periodic
orbits around Mars.
\end{abstract}

\tableofcontents

\section{Introduction}

\subsection{Motivation and general perspective}

On a merry-go-round fictitious forces appear.
One of these fictitious forces is the centrifugal force.
The centrifugal force only depends on position and is a conservative
force, i.e. it is the gradient of a potential.
If one moves on a merry-go-round an additional fictitious force
is the Coriolis force. In contrast to the centrifugal force, the
Coriolis force depends linearly on the velocity.
There is a third fictitious force, the Euler force.
The Euler force only appears if the merry-go-round
is accelerated or decelerated.
As the centrifugal force the Euler force only depends on position,
but different from the centrifugal force the Euler force is not a
conservative force, i.e. it is not a gradient vector field.

\medskip
\textsc{Coriolis force -- Twisted symplectic form.}
It was observed by {\Arnold}~\cite{Arnold:1961a}
that velocity dependent forces, as the Lorentz force of a magnetic
field or the Coriolis force, can be modeled symplectically by
twisting the standard symplectic form on the cotangent bundle.
In this article we address the question how to model the Euler force
symplectically.
In an accelerated merry-go-round the Coriolis force
itself is time-dependent and therefore modeling it as {\Arnold} did,
the symplectic form on the cotangent bundle gets time-dependent.

\smallskip
\textsc{Euler force -- Time-dependent primitive.}
We explain in this paper that the Euler force can be modeled 
with the help of a time-dependent primitive of the time-dependent
symplectic form.
In particular, in contrast to the Coriolis force the choice of the
time-dependent primitive matters for the Euler force. This is vaguely
reminiscent of the Aharonov-Bohm effect which not only depends on the
magnetic field, but also on the choice of the vector potential.
On the other hand, we do not see a direct connection between
the Aharonov-Bohm effect and the Euler force.
The Aharonov-Bohm effect is a quantum mechanical phenomenon,
while what we discuss here is a classical phenomenon where
time-dependence of the magnetic field is crucial.

\medskip
\textsc{Gateway around Mars -- Time-periodic symplectic form.}
The study of accelerated merry-go-rounds is motivated
by applying symplectic methods to find a gateway around Mars.
This general perspective is explained in our current 
work~\cite{Frauenfelder:2026b,Frauenfelder:2026d}.
In fact, the orbit of Mars is rather eccentric.
Therefore the dynamics around Mars has to be modeled by the elliptic
restricted three-body-problem and not just the circular one.
In the elliptic restricted three-body-problem the rotational velocity
is not constant, like in accelerating and decelerating merry-go-rounds.
However, time dependence of acceleration and deceleration is periodic
with period 1 Martian year.
\\
With this motivation in mind we are interested in detecting periodic
orbits in periodically accelerating and decelerating merry-go-rounds.
For that purpose we are looking at symplectic forms
depending periodically on time, together with a choice of
time-dependent primitives.
The time dependence of the primitive itself does not need to be
periodic, but twisted-periodic in a sense specified in this paper
in Definition~\ref{def:tw-per}.

\smallskip
\textsc{Collisions -- Delay equation.}
Due to collisions, the Hamiltonians in the restricted
three-body-problems are singular.
However, two-body collisions can be regularized.
If the Hamiltonian is time-dependent, as in the elliptic restricted
three-body-problem, then there is no preserved energy
so that instead of blowing up the energy hypersurface
one has to blow up the loop space, as first discovered by
Barutello, Ortega, and  Verzini~\cite{Barutello:2021b}.
After regularizing using a blow-up of the loop space, 
the symplectic form as well as the equations become non-local.
This more general perspective, how to deal with non-local symplectic
forms, is the topic of our current
work~\cite{Frauenfelder:2026b,Frauenfelder:2026d}.

\subsection{Outline and main results}

In Section~\ref{sec:merry}
we treat a merry-go-round from the Hamiltonian point of view
and derive its Hamilton equation in phase space and the force
equation in configuration space.
In the force equation we will discover the three fictitious forces,
namely the centrifugal force, the Coriolis force, and the Euler force,
as illustrated by Figure~\ref{fig:fig-Coriolis}.
This force equation fits into a more general class of equations which
we refer to as the $(\AAA,\phi)$-equation.
Here $\AAA$ is a time-dependent vector potential of a time-dependent
magnetic field while $\phi$ is a time-dependent scalar potential.

\smallskip
In Section~\ref{sec:Lagrangian}
we are considering periodic solutions of the
$(\AAA,\phi)$-equation.
In order to find periodic solutions to the $(\AAA,\phi)$-equation
one needs some periodicity assumptions on our data.
We do not need that the vector potential $\AAA$ itself is periodic in time,
but only its differentials with respect to time as well as with
respect to space.
It is possible to trade-off the potential $\phi$
with a twist in the periodicity by defining a new vector potential
$\AAA^\phi$ which is not periodic in time, but twisted-periodic.
\\
We are then discussing how periodic solutions can be detected
variationally with the help of an action functional.
Since we do not assume that the vector potential itself is periodic,
an additional term arises in the action functional which we refer to
as the \emph{twist term}.
The twist term vanishes in the periodic~case.

\smallskip
While in Section~\ref{sec:Lagrangian} we discuss
a Lagrangian approach to periodic solutions of the
$(\AAA,\phi)$-equation, we study in Section~\ref{sec:Ham} a
Hamiltonian approach.
This Hamiltonian approach fits onto a more general class of
functionals which we discuss in more detail in
Section~\ref{sec:Euler-VF}.

\begin{figure}
  \centering
  \includegraphics
                             {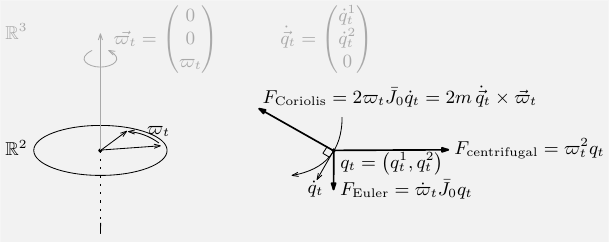}
  \caption{The three fictitious forces in the merry-go-round}
   \label{fig:fig-Coriolis}
\end{figure}

\smallskip
In Section~\ref{sec:Euler-VF}
we are considering a time-dependent family of primitives of
time-dependent symplectic forms.
Again this family does not need to depend periodically on time,
but twisted-periodically whose precise meaning is explained
in~(\ref{eq:tw-per}).
For such a twisted-periodic family of primitives we show how to
associate a well defined action functional.
We derive the critical point equation of this functional in two ways.
In the proof of Theorem~\ref{thm:Euler}
we use Cartan's formula on the product of the underlying
manifold with the circle.
In Section~\ref{sec:Cartan}
we give a different derivation of the critical point equation using
Cartan's formula on the loop space.
Since the loop space is infinite dimensional,
strictly speaking there is no mathematical theory yet
to use Cartan's formula there.
\\
It turns out that a new vector field $Y$ is entering the critical
point equation which, similarly as the Hamiltonian vector field $X$,
is implicitly defined with the help of the time-derivative
$\dot \lambda_t$ of the primitive.
We refer to this vector field $Y$ as the \emph{Euler vector field} in
view of its close relation with the Euler force.
In the second interpretation using the Cartan formula on the loop space
this derivative appears as the Lie derivative with respect to the
vector field on the free loop space which generates the rotation.

\medskip\noindent
{\bf Acknowledgements.}
UF~acknowledges support by DFG grant
FR~2637/5-1.

\section{Merry-go-round and fictitious forces}
\label{sec:merry}

We consider a free particle in the plane.
In an inertial system the free particle moves with constant velocity
along a straight line.
In particular, the Hamiltonian is just given by kinetic energy $T$.
There are no accelerations, hence no forces acting.
This changes as soon as we rotate the coordinate system with
time-dependent angular velocity
$\varpi\colon\R\to\R$. In this case apart from kinetic energy $T$
the particle has angular momentum and three
pseudo-forces appear, as illustrated by Figure~\ref{fig:fig-Coriolis},
namely
\begin{itemize}\setlength\itemsep{0ex}
\item
  \textbf{centrifugal force} depending on the locus $q$ and the angular
  velocity $\varpi_t$ is a gradient, hence conservative;
\item
  \textbf{Coriolis force} depending on the particle
  velocity $\dot q$ and $\varpi_t$;
\item
  \textbf{Euler force} depending on the locus $q$ and the
  angular acceleration $\dot \varpi_t$.
\end{itemize}

Let $\varpi\colon\R\to\R$ be a smooth function.
The function whose Hamiltonian vector field generates rotation
is angular momentum.
Hence in a rotating coordinate system with time-dependent angular
speed $\varpi_t$, say the $xy$-plane rotating anti-clockwise in space,
the Hamiltonian of a free particle, say of mass $m=1$,
consists of kinetic energy minus angular momentum, in symbols
\begin{equation*}
\begin{split}
   H^\varpi\colon\R^5
   \to\R
   ,\quad
   (t,q,\ppp)
   &\mapsto
   \tfrac12 \Abs{\ppp}^2
   -
   \INNER{
   \begin{pmatrix}0\\0\\\varpi_t \end{pmatrix}
   }
   {
\underbrace{
   \begin{pmatrix} q_1\\q_2\\0 \end{pmatrix}
   \times
   \begin{pmatrix} p_1\\p_2\\0\end{pmatrix}
}_{\text{angular momentum}}
   }
\\
   &\quad\;
   =\tfrac{p_1^2+p_2^2}{2}
   -\varpi_t(q_1p_2-q_2p_1)
   =:H^{\varpi_t}(q,\ppp)
\end{split}
\end{equation*}
where $\inner{\cdot}{\cdot}$ is the Eucliean inner product on $\R^3$
and $\abs{\cdot}$ is the induced norm.
After completing the squares $H^{\varpi}$ is of the form
\begin{equation*}
\begin{split}
   H^{\varpi_t}(q,\ppp)
   &=\underbrace{\tfrac12\left(\left(p_1+\varpi_t q_2\right)^2
   +\left(p_2-\varpi_t q_1\right)^2\right)}_{T(\ppp-\AAA_t|_q)}
   \underbrace{-\tfrac12 \varpi_t^2\Abs{q}^2}_{+\phi_t(q)}
\end{split}
\end{equation*}
where the time-dependent vector potential and the potential are
\begin{equation}\label{eq:AAA-phi}
   \AAA_t|_q=\AAA_{\varpi_t}|_q
   =\begin{pmatrix}-\varpi_t q_2\\\varpi_tq_1\end{pmatrix}
   =\varpi_t J_0 q
   ,\qquad
   \phi_t|_q=-\tfrac12 \varpi_t^2\Abs{q}^2 .
\end{equation}
The anti-/clockwise quarter rotation 
in the plane is encoded by the matrices
$$
    J_0:=\begin{pmatrix}0&-1\\1&0\end{pmatrix}
   ,\qquad
   \Jbar_0 :=-J_0=\begin{pmatrix}0&1\\-1&0\end{pmatrix}
   ,
$$
where $J_0$ is the canonical complex structure on $\R^2$.
The \textbf{Hamilton equations} for $H^\varpi$ with respect to the
\textbf{canonical symplectic form} $\omegacan=d\lambdacan$ are
\begin{equation}\label{eq:H-eqs}
  \left\{
   \begin{aligned}
     \dot q
     &=\p_p H^{\varpi_t}
     &&=\ppp-\AAA_t|_q
\\
     \dot\ppp
     &=-\p_q H^{\varpi_t}
     &&=\sum_{j=1}^2(p_j-A_t^j)\Nabla{}\AAA_t^j|_q
     -\Nabla{}\phi_t|_q
     .
   \end{aligned}
   \right.
\end{equation}
This is a first order ODE for smooth maps $(q,\ppp)\colon\R\to\R^4$.
We eliminate $\ppp$ to get a second order ODE in $q$.
For the following calculation we simplify notation $A_1:=A_t^1$ and
$A_2:=A_t^2$ as well as $\p_1:=\p_{q_1}$ and $\p_2:=\p_{q_2}$.
Differentiating the first Hamilton equation
with respect to $t$ and then using the second one for $\dot\ppp$
we obtain for $q$ the second order ODE
\begin{equation*}
\begin{split}
   \ddot q
   &=\dot \ppp-\tfrac{d}{dt} \AAA_t|_q
\\
   &=-\p_q H^{\varpi_t}(q,\ppp)
   -\AAAdot_t|_q-d\AAA_t|_q\,\dot q
\\
   &=-\p_q\left(\tfrac12 \Abs{\ppp-\AAA_t(q)}^2\right)
   -\Nabla{}\phi_t|_q
   -\AAAdot_t|_q-d\AAA_t|_q\,\dot q
\\
   &=-\begin{pmatrix}
      \p_1 \tfrac{(p_1-A_1)^2+(p_2-A_2)^2}{2} \\
      \p_2 \tfrac{(p_1-A_1)^2+(p_2-A_2)^2}{2}
   \end{pmatrix}
   -\Nabla{}\phi_t
   -\AAAdot_t
   -\begin{pmatrix}
      \p_1A_1 & \p_2A_1 \\
      \p_1A_2 & \p_2A_2
   \end{pmatrix}
   \begin{pmatrix}
      \dot q_1\\\dot q_2
   \end{pmatrix}
\\
   &=
   \small
   \begin{pmatrix}
      \underline{(p_1-A_1)\p_1A_1}+(p_2-A_2)\p_1A_2 \\
      (p_1-A_1)\p_2A_1+\underline{ \underline{ (p_2-A_2)\p_2A_2}}
   \end{pmatrix}
   -\begin{pmatrix}
      \underline{ (\p_1A_1)\dot q_1}+(\p_2A_1)\dot q_2 \\
      (\p_1A_2)\dot q_1+\underline{ \underline{ (\p_2A_2)\dot q_2}}
   \end{pmatrix}
   -\Nabla{}\phi_t
   -\AAAdot_t
\\
   &=\begin{pmatrix}
      (\p_1A_2-\p_2A_1)\dot q_2 \\
      (\p_2A_1-\p_1A_2)\dot q_1
   \end{pmatrix}
   -\Nabla{}\phi_t
   -\AAAdot_t
   \qquad
   {\color{gray},\,\p_1A_2-\p_2A_1=:\rot\,\AAA_t}
\\
   &=\left(\rot\, \AAA_t\right)
   \begin{pmatrix} 0&1\\-1&0 \end{pmatrix}
   \begin{pmatrix} \dot q_1\\\dot q_2 \end{pmatrix}
   -\Nabla{}\phi_t
   -\AAAdot_t
\\
   &=-\left(\rot\, \AAA_t|_q\right) J_0\dot q-\AAAdot_t|_q
   -\Nabla{}\phi_t|_q
{\color{gray}\;
   =\ddot q
   .
}
\end{split}
\end{equation*}
Underlined terms cancel due to the first Hamilton equation.
Thus~(\ref{eq:H-eqs}) $\Rightarrow$~(\ref{eq:q}).

\begin{lemma}\label{le:qp-q}
The first order ODE~(\ref{eq:H-eqs})
is equivalent to the second order~ODE
\begin{equation}\label{eq:q}
   \ddot q
   =
   -\left(\rot\,\AAA_t|_q\right) J_0\dot q
   -\AAAdot_t|_q
   -\Nabla{}\phi_t|_q .
\end{equation}
We call~(\ref{eq:q}) the \textbf{\boldmath$(\AAA,\phi)$-equation}.
It makes sense for
smooth $\phi\colon\R\times\R^2\to\R$.
\end{lemma}

\begin{proof}
We already showed that~(\ref{eq:H-eqs}) $\Rightarrow$~(\ref{eq:q}).
Vice versa, let $q$ solve~(\ref{eq:q})
and define $\ppp:=\dot q+\AAA_t|_q$. Hence
$\dot\ppp=\ddot q+\AAAdot_t|_q+d\AAA_t|_q \dot q$.
Now read the earlier displayed calculation backwards to see that
$\ddot q=-\p_q H^{\varpi_t}(q,\ppp)-\AAAdot_t|_q-d\AAA_t|_q\,\dot q$.
Hence $\dot\ppp=-\p_q H^{\varpi_t}(q,\ppp)$ and this
proves~(\ref{eq:H-eqs}) and Lemma~\ref{le:qp-q}.
\end{proof}

\begin{corollary}\label{cor:merry}
In the merry-go-round case ($\AAA$, $\phi$
given by~(\ref{eq:AAA-phi})) we have 
$$
   \rot\,\AAA_t|_q
   :=\p_1 A_t^2|_q-\p_2 A_t^1|_q
   =2\varpi_t
   ,\qquad
   \AAAdot_t|_q
   =\dot\varpi_t J_0 q
   ,\qquad
   \Nabla{}\phi|_q
   =-\varpi_t^2 q ,
$$
so that the $(\AAA,\phi)$-equation becomes
\begin{equation*}
\begin{split}
   \ddot q
   &=-2\varpi_t J_0\dot q
   -\dot\varpi_t J_0 q
   +\varpi_t^2 q
\\
   &=\underbrace{2\varpi_t \Jbar_0\dot q}_{\rm Coriolis}\;
   +\;\underbrace{\dot\varpi_t \Jbar_0 q}_{\rm Euler}\;
   +\underbrace{\varpi_t^2 q}_{\rm centrifug.} .
\end{split}
\end{equation*}
\end{corollary}

\newpage
\boldmath
\section[Lagrangian variational approach to periodic solutions]
{Lagrangian variational approach to periodic solutions of the $(\AAA,\phi)$-equation}
\label{sec:Lagrangian}
\unboldmath

In this section we are looking at periodic solutions of the
equation~(\ref{eq:q}), namely
$$
   \ddot q
   =
   -\left(\rot\,\AAA_t|_q\right) J_0\dot q-\AAAdot_t|_q
   -\Nabla{}\phi_t|_q .
$$
In this article \textbf{periodic} means $1$-periodic
and we identify $\SS^1$ with $\R/\Z$.
To make sense of periodicity of solutions of the
$(\AAA,\phi)$-equation we need that all three $\left(\rot\,\AAA_t\right)$,
$\AAAdot_t$, and $\phi_t$ are periodic in time $t$.
We do not need that $\AAA_t$ itself is periodic in $t$, but we need
some twisted-periodicity as explained next.

\begin{definition}[Twisted-periodic $1$-form]
\label{def:twisted-1-form}
Let $\Qfrak\subset\R^2$ be an open subset.
A \textbf{twisted-periodic \boldmath$1$-form}
$\theta=\{\theta_t\}_{t\in\R}$
is a smooth real family of $1$-forms on $\Qfrak$
\begin{equation*}
   \theta_t=A_t^1\, dq_1+A_t^2\, dq_2
   ,\qquad
   \AAA_t=\left(A_t^1, A_t^2\right)\colon\Qfrak\to\R^2
   ,
\end{equation*}
such that $\theta$ is \textbf{twisted-periodic} in the sense that a)~the
$t$-derivative~is~periodic
\begin{equation}\label{eq:df}
   \text{a)~$\dot\theta_{t+1}=\dot\theta_t$}
   ,\qquad
   \text{b)~$\theta_{t+1}=\theta_t+df_t$}
   ,
\end{equation}
and b)~for every time $t$ there is a smooth so-called
\textbf{twist function} $f_t\colon\Qfrak\to \R$.
A \textbf{time-constant} choice, e.g. $f:=f_0$,
is a twist function for $\theta$, too;
cf.~(\ref{eq:tw-per}).
\end{definition} 

Consequently the exterior derivative is periodic, in symbols
$
   d\theta_{t+1}=d\theta_t
$.
Furthermore, taking the time derivative of~(\ref{eq:df})$_{\mathrm{b}}$
we get $d\dot f_t=\dot\theta_{t+1}-\dot\theta_t=0$.
Thus $\dot f_t(q)$ is \textbf{locally constant} in $q$,
not necessarily in~$t$.
If $\theta_{t+1}=\theta_t$ is periodic, then the zero function is a
twist function, indeed $d\theta_{t+1}=d\theta_t+d 0$.
Viewing the plane as a subspace of space $\R^2\times\{0\}\subset \R^3$
we use the coefficients of the twisted $1$-form $\theta$ on the plane
to define a vector field in space by
 $$
   \left(\AAA_t,0\right)|_{(q_1,q_2,q_3)}:=\left(A_t^1|_{(q_1,q_2)},
   A_t^2|_{(q_1,q_2)},0\right)
$$
called a \textbf{vector potential}.

\begin{remark}
Magnetic fields are described by closed $2$-forms $\sigma$,
where closedness $d\sigma=0$ encodes the fact that there are no
magnetic charges; see e.g.~\cite[\S\,2.4.1]{Weber:2017b}.
For a twisted-periodic $1$-form $\theta$ the family of $2$-forms
\begin{equation}\label{eq:d-theta}
\begin{split}
   \sigma_t:=d\theta_t
   &=d\left(A_t^1\, dq_1+A_t^2\, dq_2\right)
   =\underbrace{\left(\p_1A^2_t-\p_2A^1_t\right)}_{=: \rot\,\AAA_t}
    dq_1\wedge dq_2
\end{split}
\end{equation}
is time-periodic 
$
   \sigma_t:=d\theta_t=d\theta_{t+1}=\sigma_{t+1}
$
and closed $d\sigma_t=dd\theta_t=0$.
The magnetic vector field corresponding to the magnetic $2$-form
is $\BBB_t:=(*\sigma)^\#$, i.e.
$$
   \BBB_t:=\underbrace{\nabla\times}_{=\,\rot}
   \begin{pmatrix}\AAA_t\\0\end{pmatrix}
   =
   \begin{pmatrix} \p_{q_1}\\\p_{q_2}\\\p_{q_3} \end{pmatrix}
   \times
   \begin{pmatrix} A_t^1\\A_t^2\\0\end{pmatrix}
   =\begin{pmatrix} -\p_{q_3} A_t^2\\\p_{q_3} A_t^1\\
      \p_1 A_t^2-\p_2 A_t^1\end{pmatrix}
   =\begin{pmatrix} 0\\0\\\rot\,\AAA_t\end{pmatrix} .
$$
This shows that the magnetic field is
perpendicular to the $q_1$-$q_2$-plane in $\R^3$.
\end{remark}

\subsection{Classical action functional}\label{sec:class-action}

\begin{definition}\label{def:class-action}
Let $\{\theta_t\}$ be a twisted-periodic $1$-form on
$\Qfrak\subset\R^2$ open.
Let $\phi\colon \SS^1\times\Qfrak\to\R$ be a smooth map.
Define the \textbf{classical action functional} by
\begin{equation*}
\begin{split}
   \Ss=\Ss_{L^\theta}\colon\Lambda\Qfrak:=W^{1,2}(\SS^1,\Qfrak)
   &\to\R
\\
   q
   &\mapsto\int^1_0
   \underbrace{\tfrac12\Abs{\dot q_t}^2
   +\theta_t|_{q_t}\dot q_t-\phi_t|_{q_t}}_{=L^\theta_t(q_t,\dot q_t)} \,
   dt
   -f_0(q_0)
   .
\end{split}
\end{equation*}
For $f_0$ see~(\ref{eq:df}).
The function $L^\theta_t\colon T\Qfrak\to\R$, called \textbf{Lagrangian},
is defined 
by
$$
   L^\theta_t(q,v)
   :=\tfrac12\Abs{v}^2+\theta_t|_{q} v-\phi_t(q)
   .
$$
\end{definition}

\begin{remark}[twist term and integration interval]
\label{rem:twist-term}

(i)~Given the integration interval $[\underline{0},1]$,
it is important to subtract the
\textbf{twist term \boldmath$f_{\underline{0}}(q_{\underline{0}})$}
to get as critical points the periodic solutions of the
$(\AAA,\phi)$-equation
(Propositions~\ref{prop:integration-HTheta}--\ref{prop:crit-Ss_L}).

(ii)~If $\theta_t$ is twisted-periodic, and not periodic, then $L_t^\theta$
is not periodic in time.
Therefore the integral defining $\Ss(q)$
is not over the circle $\SS^1$,
but defined over the interval $[0,1]$.
Given $r\in\R$, using $\int_r^{r+1}$ in the definition of $\Ss(q)$
causes that $\Ss(q)$ depends on $r$ modulo $1$:
indeed shift $t$ by $1$ to obtain equality one
\begin{equation*}
\begin{split}
   \int_{r+1}^{r+2}\theta_t|_{q_t}\dot q_t \, dt
   -\int_r^{r+1}\theta_t|_{q_t}\dot q_t \, dt
   &\stackrel{{\color{gray}1}}{=}
   \int_{r}^{r+1}
   \underbrace{\theta_{t+1}}_{\theta_{t}+df_t}
   |_{\underbrace{q_{t+1}}_{q_t}}
   \underbrace{\dot q_{t+1}}_{\dot q_{t}} \, dt
   -\int_r^{r+1}\theta_t|_{q_t}\dot q_t \, dt
\\
   &\stackrel{{\color{gray}2}}{=}
   \int_{r}^{r+1}
   \left(
   \left.\tfrac{d}{dt}\right|_{0} f_t(q_t)
   -\dot f_t(q_t)
   \right)
   dt
\\
   &\stackrel{{\color{gray}3}}{=}
   f_{r+1}(q_{r+1})
   -f_r(q_r)
   -\int_{r}^{r+1}\underbrace{\dot f_t(q_t)}_{\dot f_t(q_{\color{red}r})}\, dt
   = 0
   .
\end{split}
\end{equation*}
After equality~1 we use periodicity $q_{r+1}=q_r$ and twisted-periodicity. 
Equality~2 cancels the two $\theta_t$ summands with opposite signs
and rewrites $df_t|_{q_t}\dot q_t$ by the chain rule.
Equalities~3 and four use the fundamental theorem of calculus.
After equality~3 we use that $\dot f_t$ is locally
constant, so we may replace $q_t$ by $q_{\color{red}r}$.

(iii)~If $\theta_{t+1}=\theta_t$ is periodic, then there is no twist
term ($f\equiv 0$) and the integrals
$\int_r^{r+1} L_t^\theta(q_t,\dot q_t)dt=\int_0^1 L_t^\theta(q_t,\dot q_t)dt$
are equal for any $r\in\R$.\footnote{
  schematically
  $\int_r^{r+1}\theta_t=\int_r^1 \theta_t+\int_1^{r+1}\theta_t
  =\int_r^1 \theta_t+\int_0^r \theta_{t+1}=\int_0^1 \theta_t$
  as $\theta_{t+1}=\theta_t$ is periodic
  }
\end{remark}

\begin{definition}[2-form $\Sigma:=D\varTheta$ on $\SS^1\times \Qfrak$]
\label{def:Theta}
A twisted-periodic $1$-form $\{\theta_t\}_{t\in\R}$ on
$\Qfrak\subset\R^2$ induces a $1$-form $\varTheta$
on $\R\times \Qfrak$ only, not $\SS^1\times \Qfrak$, namely
\begin{equation*}
   \varTheta|_{(t,x)}
   \begin{pmatrix}\tau\\\xi\end{pmatrix}
   :=\theta_t|_x\xi
   \qquad
   \forall (t,x)\in \R\times \Qfrak
   \quad
   \forall (\tau,\xi)\in \R\times T_x \Qfrak
   .
\end{equation*}
The exterior derivative on $\R\times \Qfrak$,
notation $D$, is given by
\begin{equation*}
\begin{split}
    D\varTheta|_{(t,x)}
   \left(\begin{pmatrix}\tau \\\xi\end{pmatrix},
   \begin{pmatrix}\sigma\\\eta\end{pmatrix}\right)
   &=d\theta_t|_x (\xi,\eta)
   +\left(dt\wedge \dot\theta_t|_x\right)
   \left(\begin{pmatrix}\tau \\\xi\end{pmatrix},
   \begin{pmatrix}\sigma\\\eta\end{pmatrix}\right)
\\
   &=d\theta_t|_x (\xi,\eta)
   +\dot\theta_t|_x\eta\cdot \tau 
   -\dot\theta_t|_x\xi\cdot \sigma .
\end{split}
\end{equation*}
By twisted periodicity~(\ref{eq:df}) this is a $2$-form
on $\SS^1\times\Qfrak$ notation $\Sigma:=D\varTheta$.
\end{definition}

The following proposition gives a different description of the
magnetic part of the functional which is useful to compute the
critical points in Section~\ref{sec:crit}.

\begin{proposition}[Magnetic part -- alternative description]
\label{prop:integration-HTheta}
Let $\theta$ and $\Sigma$ be as in Definition~\ref{def:Theta}.
Fix a base loop $b\in\Lambda\Qfrak$.
Then every loop $q$ homotopic to $b$ in $\Lambda\Qfrak$
through a smooth homotopy, say $h$ from $b$ to $q$,
satisfies the identity
$$
   \int_{[0,1]} q^*\theta
   -f_0(q_0)
   =\int_R H^*\Sigma
   +\int_{[0,1]} b^*\theta
   -f_0(b_0)
$$
with $H\colon R=[0,1]\times[0,1]\to \SS^1\times\Qfrak$,
$(r,t)\mapsto\left(t,h(r,t)\right)$ and $f_0$ from~(\ref{eq:df}).
\end{proposition}

\begin{proof}
Let $h\colon Z=[0,1]\times \SS^1\to\Qfrak$,
$(r,t)\mapsto h(r,t)$, be a smooth homotopy of loops from
$b=h(0,\cdot)$ to $q=h(1,\cdot)$.
Consider the induced homotopy $H$ between loops $B$ and $Q$ in
$\SS^1\times\Qfrak$ defined by
\begin{equation}\label{eq:H-Q-B}
   H(r,t):=\begin{pmatrix}t\\h(r,t)\end{pmatrix}
   ,\quad
   B_t:=\begin{pmatrix}t\\b_t\end{pmatrix}=H(0,t)
   ,\quad
   Q_t:=\begin{pmatrix}t\\q_t\end{pmatrix}=H(1,t)
   .
\end{equation}
We use Definition~\ref{def:Theta} of $\varTheta$ in each of the following
two calculations
\begin{equation}\label{eq:Q}
\begin{split}
   \int_{[0,1]} Q^*\varTheta
   =\int_0^1
   \varTheta|_{(t,q_t)}\begin{pmatrix}1\\\dot q_t\end{pmatrix}
   \, dt
   =
   \int_0^1
   \theta_t|_{q_t} \dot q_t
   \, dt
   =\int_{[0,1]} q^*\theta
\end{split}
\end{equation}
and
\begin{equation*}
\begin{split}
   \int_{[0,1]} H(\cdot,1)^*\varTheta
   =\int_0^1
   \varTheta|_{(1,h(r,1))}\begin{pmatrix}0\\\p_1h(r,1)\end{pmatrix}
   \, dr
   &=
   \int_0^1
   \theta_1|_{h(r,1)} \p_1h(r,1)
   \, dr
\\
   &=\int_{[0,1]} h(\cdot, 1)^*\theta_1
   .
\end{split}
\end{equation*}
Similarly for $B^*\varTheta$ and $H(\cdot,0)^*\varTheta$.
%
%
The rectangle boundary $\p R$ has four 
segments
$$
   \Gamma_1=\{1\}\times [0,1]_t
   ,\;\;
   \Gamma_2=[0,1]_r\times \{1\}
   ,\;\;
   \Gamma_3=\{0\}\times [0,1]_t
   ,\;\;
   \Gamma_4=[0,1]_r\times \{0\}
   ,
$$
endowed with the induced (anti-clockwise) boundary orientation.
Exactness $\Sigma=d\varTheta$ and as
exterior differentiation $D$ and pull-back commute, Stokes yields
\begin{equation*}
\begin{split}
   \int_{R} H^*\Sigma
   &=  
   \int_{\p R=\Gamma_1\cup \Gamma_2\cup \Gamma_3\cup \Gamma_4} H^*\varTheta
\\
   &= 
   \int_{\Gamma_1} Q^*\varTheta
   +\int_{\Gamma_2} H(\cdot,1)^*\varTheta
   +\int_{\Gamma_3} B^*\varTheta
   +\int_{\Gamma_4} H(\cdot,0)^*\varTheta
\\
   &\stackrel{{\color{gray}3}}{=}
   \int_{t=0}^1q^*\theta
   -\int_{r=0}^1
      (\underbrace{h(\cdot,1)}_{=h(\cdot,0)})^*\theta_1
   -\int_{t=0}^1b^*\theta
   +\int_{r=0}^1 h(\cdot,0)^*\theta_0
\\
   &\stackrel{{\color{gray}4}}{=}
   \int_{t=0}^1q^*\theta
   -\int_{t=0}^1b^*\theta
   -\int_{r=0}^1 
   \underbrace{h(\cdot,0)^*\left(\theta_1-\theta_{\underline{0}}\right)}
      _{=h(\cdot,0)^* df_{\underline{0}}=d h(\cdot,0)^* f_{\underline{0}}}
\\
\end{split}
\end{equation*}
\begin{equation*}
\begin{split}
   &\stackrel{{\color{gray}5}}{=}
   \int_{t=0}^1q^*\theta
   -\int_{t=0}^1b^*\theta
   -f_{\underline{0}}\circ\underbrace{h(1,0)}_{q_0}
   +f_{\underline{0}}\circ\underbrace{h(0,0)}_{b_0}
   .
\end{split}
\end{equation*}
Equality~3 parametrizes the four segments 
$\Gamma_1,\dots, \Gamma_4$ by $[0,1]$
taking care of orientations, then the previously prepared formulas
such as for $\int Q^*\varTheta$ are used.
Note that $h(\cdot,1)\equiv h(\cdot,0)$ since $t\in\SS^1$.
Equality~4 is by linearity of integral and pull-back.
Now twisted-periodicity~(\ref{eq:df}) is used
and $d$ and pull-back are interchanged.
Equality~5 is by Stokes' theorem.
This proves Proposition~\ref{prop:integration-HTheta}.
\end{proof}

\subsection{Critical points}\label{sec:crit}

\begin{proposition}\label{prop:crit-Ss_L}
The critical points of the classical action functional
$\Ss\colon\Lambda\Qfrak\to\R$, Definition~\ref{def:class-action},
are the solutions of the $(\AAA,\phi)$-equation, namely
\begin{equation*}
   \ddot q=-\left(\rot\,\AAA_t|_q\right) J_0\dot q
      -\AAAdot_t|_q -\Nabla{}\phi_t|_q
\end{equation*}
for smooth loops $q\colon\SS^1\to\Qfrak$.
\end{proposition}

\begin{proof}
Pick $q\in \Lambda\Qfrak$.
Fix a base loop $b\in \Lambda\Qfrak$ smoothly homotopic 
to $q$ in $\Lambda\Qfrak$.
We move from $\Qfrak$ to $\SS^1\times \Qfrak$
and define $\varTheta$ and $D\varTheta$
as in Definition~\ref{def:Theta}.
Loops $b$ and $q$ in $\Qfrak$ induce loops $B$ and $Q$ in $\SS^1\times
\Qfrak$, and a vector field $\xi$ along $q$ induces a vector field $\Xi$
along $Q$, namely
$$
   B_t=\begin{pmatrix}t\\b_t\end{pmatrix}
   ,\qquad
   Q_t=\begin{pmatrix}t\\q_t\end{pmatrix}
   ,\qquad
   \dot Q_t =\begin{pmatrix}1\\\dot q_t\end{pmatrix}
   ,\qquad
   \Xi=\begin{pmatrix}0\\\xi\end{pmatrix}
   .
$$
Let $h\colon Z=[0,1]\times \SS^1\to\Qfrak$,
$(r,t)\mapsto h(r,t)$, be a smooth homotopy of loops from
$b=h(0,\cdot)$ to $q=h(1,\cdot)$.
Consider the induced homotopy
$H\colon [0,1]\times[0,1]\to \SS^1\times\Qfrak$
between loops $B$ and $Q$ in
$\SS^1\times\Qfrak$ defined by~(\ref{eq:H-Q-B}).

\smallskip\noindent
Now we construct variations with real parameter $\tau$.
We vary the point $q$ by $q_{(\tau)}:=q+\tau\xi$ and
we add to the homotopy $h$ from $b$ to $q$
the linear homotopy from $q$ to $q_{(\tau)}$ extending the parameter
interval to $r\in[0,1+\tau]$, in symbols
\begin{equation}\label{eq:h}
   h_{(\tau)}(r,\cdot)
   :=
   \begin{cases}
      h(r,\cdot)&\text{, $r\in[0,1]$,}\\
      q+(r-1)\xi &\text{, $r\in[1,1+\tau]$.}
   \end{cases}
\end{equation}
We lift $q_{(\tau)}$ to $Q_{(\tau)}=Q+\tau\Xi=(\cdot,q+\tau\xi)$
and $h_{(\tau)}$ to $H_{(\tau)}(r,t):=(t,h_{(\tau)}(r,t))$.

\medskip\noindent
For $\Sigma=D\varTheta$ from Definition~\ref{def:Theta}
we determine the derivative of the magnetic part
in the alternative description of
Proposition~\ref{prop:integration-HTheta}.
We compute, using in equality~1 that only the first summand depends
on $\tau$, the following $\tau$-derivative
\begin{equation*}
\begin{split}
   &\left.\frac{d}{d\tau}\right|_{0}
   \biggl(
   \int_{[0,1+\tau]\times\SS^1}H_{(\tau)}^*\Sigma+\int_{[0,1]} b^*\theta-f_0(b_0)
   \biggr)
\\
   &\stackrel{{\color{gray}1}}{=}
   \lim_{\tau\to0}\frac{1}{\tau}
   \biggl (
    \int_{[0,1+\tau]\times\SS^1}H_{(\tau)}^*\Sigma- \int_{[0,1]}H_{(\tau)}^*\Sigma
   \biggr)
\\
   &=\lim_{\tau\to0}\frac{1}{\tau}
   \int_{[1,1+\tau]\times\SS^1}H_{(\tau)}^*\Sigma
\\
   &=\lim_{\tau\to0}\frac{1}{\tau}
   \int_{r=1}^{1+\tau}\int_{t=0}^1
   \Sigma_{H_{(\tau)}(r,t)}\left(\p_r H_{(\tau)}(r,t),\p_t H_{(\tau)}(r,t))\right) dt\, dr
\\
\end{split}
\end{equation*}
\begin{equation*}
\begin{split}
   &=
   \lim_{\tau\to0}\frac{1}{\tau}
   \int_{r=1}^{1+\tau}\int_{t=0}^1
   D\varTheta_{(t,q_t+(r-1)\xi_t)}\left(\begin{pmatrix}0\\\xi_t\end{pmatrix},
   \begin{pmatrix}1\\\dot q_t+(r-1)\dot\xi_t\end{pmatrix}\right) dt\, dr
\\
   &\stackrel{{\color{gray}5}}{=}
   \lim_{\tau\to0}\frac{1}{\tau}
   \int_{r=1}^{1+\tau}\int_{0}^1
   \left(
   d\theta_t|_{q_t+(r-1)\xi_t}\bigl(\xi_t,\dot q_t+(r-1)\dot\xi_t\bigr)
   -\dot\theta_t|_{q_t+(r-1)\xi_t}\xi_t
   \right)
   dt\, dr
\\
   &\stackrel{{\color{gray}6}}{=}
   \lim_{\tau\to0}\frac{1}{\tau}
   \int_{\rho=0}^1 \int_{t=0}^1
   \left(
   d\theta_t|_{q_t+\rho\tau\xi_t}\bigl(\xi_t,\dot q_t+\rho\tau\dot\xi_t\bigr)
   -\dot\theta_t|_{q_t+\rho\tau\xi_t}\xi_t
   \right)
   \tau\, dt\, d\rho
\\
   &\stackrel{{\color{gray}7}}{=}
   \int_{\rho=0}^1 \int_{t=0}^1
   \left(
   d\theta_t|_{q_t}\bigl(\xi_t,\dot q_t\bigr)
   -\dot\theta_t|_{q_t}\xi_t
   \right)
   dt\, d\rho
\\
   &\stackrel{{\color{gray}8}}{=}
   \int_{t=0}^1
   \left(
   d\theta_t|_{q_t}\bigl(\xi_t,\dot q_t\bigr)
   -\dot\theta_t|_{q_t}\xi_t
   \right)
   dt
   .
\end{split}
\end{equation*}
Equality~5 uses the formula for $D\varTheta$ in
Definition~\ref{def:Theta}.
Equality~6 is by variable substitution
$\rho(r):=\tfrac{r-1}{\tau}$, hence $dr=\tau\, d\rho$.
This extra factor $\tau$ and $\frac{1}{\tau}$ cancel,
so in equality~7 we can set $\tau=0$.
Equality~8 uses that the integrand does not depend on $\rho$
and that $\int_0^1d\rho=1$.

\medskip\noindent
Next we use Proposition~\ref{prop:integration-HTheta}
for $q_{(\tau)}:=q+\tau\xi$ and $q_{(\tau)}$
in equation~3 and the result of the
previous calculation in equation~4 of what follows
\begin{equation*}
\begin{split}
   &d\Ss|_{q}\xi
\\
   &=\left.\tfrac{d}{d\tau}\right|_{\tau=0} \Ss(q+\tau\xi)
\\
   &=\int_0^1 \left.\tfrac{d}{d\tau}\right|_{0}
   \left(\tfrac12\abs{\dot q_t+\tau\dot \xi_t}^2
   -\phi_t(q_t+\tau\xi_t)
   \right) dt
   +
   \left.\tfrac{d}{d\tau}\right|_{0}
    \biggl(
   \int_{[0,1]} q_{(\tau)}^*\theta
   -f_0(q_{(\tau)}(0))
   \biggr)
\\
   &\stackrel{{\color{gray}3}}{=}
   \int_0^1 \left(\inner{\dot q_t}{\dot \xi_t}
   -d\phi_t|_{q_t}\xi_t
   \right) dt
   +\left.\tfrac{d}{d\tau}\right|_{0}
   \biggl(
   \int_R H_{(\tau)}^*\Sigma+\int_{[0,1]} b^*\theta-f_0(b_0)
   \biggr)
\\
   &\stackrel{{\color{gray}4}}{=}
   \int_0^1 \left(\inner{\dot q_t}{\dot \xi_t}
   -\inner{\Nabla{}\phi_t|_{q_t}}{\xi_t}
   +d\theta_t|_{q_t} (\xi,\dot q_t)
   -\dot\theta_t|_{q_t}\xi
   \right) dt
\\
   &\stackrel{{\color{gray}5}}{=}
   \int_0^1 \Bigl(\inner{\dot q_t}{\dot \xi_t}
   -\inner{\Nabla{}\phi_t|_{q_t}}{\xi_t}
   +\left(\rot\,\AAA_t\right)
   \underbrace{\left(dq_1\wedge dq_2\right) (\xi,\dot q)}
      _{=\xi_1\dot q_2-\xi_2\dot q_1}
   -\dot\theta_t|_{q_t}\xi
   \Bigr) dt
\\
   &\stackrel{{\color{gray}6}}{=}
   \int_0^1 \inner{\dot q_t}{\dot \xi_t} dt
   +\INNER{\xi}{-\Nabla{}\phi_t|_q-\left(\rot\,\AAA_t\right) J_0\dot q
      -\AAAdot_t}_{L^2(\SS^1,\R^2)} .
\end{split}
\end{equation*}
Equality~5 replaces $d\theta_t$ by~(\ref{eq:d-theta}).
Equality~6 uses $\AAA_t$ from
Definition~\ref{def:twisted-1-form}.

Now suppose that $q$ is a critical point of $\Ss$,
i.e. $d\Ss|_{q}=0$.
Then the identity we just proved tells that
$\dot q$ has a weak derivative, notation $\ddot q$, and
$$
   \ddot q=-\left(\rot\,\AAA_t|_q\right) J_0\dot q
      -\AAAdot_t|_q -\Nabla{}\phi_t|_q .
$$
Note that the right hand side is in $L^2$, hence $q\in W^{2,2}(\SS^1,\R^2)$.
But then the right hand side is in $W^{1,2}$, hence $q\in
W^{3,2}(\SS^1,\R^2)$ and so on.
Therefore $q\in \cap_{\ell\in\N_0 }W^{\ell,2}(\SS^1,\R^2)=C^\infty(\SS^1,\R^2)$.
This proves Proposition~\ref{prop:crit-Ss_L}.
\end{proof}

\boldmath
\subsection{Eliminate periodic scalar potentials}
\label{sec:scalar}
\unboldmath

\begin{definition}[Twisted-periodic vector potential]
\label{def:twisted-vp}
A smooth map $\AAA\colon\R\times\Qfrak\to\R^2$,
$(t,q)\mapsto \AAA(t,q)=:\AAA_t|_q$, is called a
\textbf{twisted-periodic vector potential on \boldmath$\Qfrak$}
%
%
if at each time $t$ it holds that
$$
   \AAAdot_{t+1}=\AAAdot_t
   ,\qquad
   \rot\,\AAA_{t+1}
   =\rot\,\AAA_t
   ,\qquad
   \AAA_{t+1}-\AAA_t=\Nabla{}f_t
   ,
$$
for a smooth function $f_t\colon\Qfrak\to\R$,
called a \textbf{twist function}.
\end{definition}

\begin{lemma}\label{le:scal-pot}
Let $\phi\colon\SS^1\times\Qfrak\to\R$ be a smooth function
and $\AAA$ a twisted-periodic
vector potential on $\Qfrak$ with twist function $f$.
The \textbf{\boldmath$(\AAA,\phi)$-equation on \boldmath$\Qfrak$}
\begin{equation}\label{eq:A-phi}
   \ddot q
   =-\left(\rot\, \AAA_t|_q\right) J_0\dot q
   -\AAAdot_t|_q
   -\Nabla{}\phi_t|_q
\end{equation}
is an ODE for smooth maps $q\colon \R\to\Qfrak$.
Then the maps defined by
$$
   (\AAA^\phi)_t
   :=\AAA_t+\int_0^t \Nabla{}\phi_s \, ds
   ,\qquad
   f_t^\phi
   :=f_t+\int_t^{t+1}\phi_s\, ds
   ,
$$
are a twisted-periodic vector potential with twist function.
Moreover, a map $q$ is a solution of the $(\AAA,\phi)$-equation
iff $q$ is a solution of the $(\AAA^\phi,0)$-equation.
\end{lemma}

\begin{proof}
To see that $(\AAA^\phi)_t$ is a twisted-periodic vector potential
note that, firstly
$$
   \dot{(\AAA^\phi)}_{t+1}
   =\AAAdot_{t+1}+\Nabla{}\phi_{t+1}
   =\underline{\AAAdot_t+\Nabla{}\phi_t
   =\dot{(\AAA^\phi)}_t}
   ,
$$
secondly,
$$
   \small
   \underbrace{\rot\,(\AAA^\phi)_{t+1}}
   =\rot\,\AAA_{t+1}+\int_0^{t+1}\underbrace{\rot(\Nabla{}}_{=0}\phi_s)\, ds
   \underbrace{=\rot\,\AAA_{t+1}}
   =\underline{\rot\,\AAA_t
   \stackrel{{\color{gray}4}}{=}\rot\,(\AAA^\phi)_t}
   .
$$
where identity~4 holds true by the underbraced one at time $t$ and,
thirdly
$$
   \AAA^\phi_{t+1}-\AAA^\phi_{t}
   =\AAA_{t+1}-\AAA_{t}+\int_t^{t+1}\Nabla{}\phi_s\, ds
   =\nabla\left(f_t+\int_t^{t+1}\phi_s\, ds\right) .
$$
Now suppose that $q$ solves~(\ref{eq:A-phi}). Then by the two
underlined identities above
$$
   -\left(\rot\, (\AAA^\phi)_t\right) J_0\dot q
   -\dot{(\AAA^\phi)}_t -0
   =\left(\rot\,\AAA_t\right) J_0\dot q
   -\left(\AAAdot_t+\Nabla{}\phi_t(q)\right)
   =\ddot q ,
$$
i.e. $q$ solves the $(\AAA^\phi,0)$-equation.
Vice versa, in the $(\AAA^\phi,0)$-equation for $q$~insert
$\p_t (\AAA^\phi)_{t}=\AAAdot_t+\Nabla{}\phi_t$
and $\rot\,(\AAA^\phi)_t=\rot\,\AAA_t$ to get (\ref{eq:A-phi})
proving Lemma~\ref{le:scal-pot}.
\end{proof}

\newpage 
\boldmath
\section{Hamiltonian variational approach to periodic solutions}
\label{sec:Ham}
\unboldmath

Let $\Qfrak\subset\R^2$ be an open subset.
Let $\phi\colon \SS^1\times\Qfrak\to\R$ be a smooth map.
Consider the Hamiltonian given by
\textbf{kinetic plus potential energy}
$$
   H\colon \SS^1\times T^*\Qfrak\to\R
   ,\quad
   (t,q,p)\mapsto \tfrac12\Abs{p}^2+\phi_t(q).
$$
Let $\{\theta_t\}$ be a twisted-periodic $1$-form on $\Qfrak$,
see~(\ref{eq:df}), in particular
\begin{equation*}
   \theta_t=A_t^1\, dq_1+A_t^2\, dq_2
   ,\qquad
   \dot\theta_{t+1}=\dot\theta_t
   ,\qquad
   \theta_{t+1}=\theta_t+df_t
   .
\end{equation*}
On the cotangent bundle $\pi\colon T^*\Qfrak\to \Qfrak$
we consider the time-dependent $1$-form
$$
   \lambda_t
   :=\lambdacan+\pi^*\theta_t
$$
where $\lambdacan$ is the Liouville form,
see e.g.~\cite[\S 37\,B]{Arnold:1989a}.
Then $\{\lambda_t\}$ satisfies twisted-periodicity~(\ref{eq:df}),
e.g. $\lambda_{t+1}=\lambda_t+ F_t$ for
$F_t(q,p):=(\pi^* f_t)(q,p)=f_t(q)$.
Furthermore, at each time $t$ the exterior derivative
$$
   \omega_t:=d\lambda_t=d\lambdacan+\pi^* d\theta_t
   ,\qquad
   \omega_{t+1}=\omega_t
   ,
$$
is a periodic \textbf{twisted symplectic form},
see e.g.~\cite[appendix]{Frauenfelder:2026b}.

\begin{definition}\label{def:symplectic-action-4}
The \textbf{symplectic action functional} is defined by
\begin{equation*}
\begin{split}
   \Aa=\Aa_{\lambda,H}\colon\Lambda T^*\Qfrak
   &\to\R
\\
   v=(q,p)
   &\mapsto
   \int_0^1\Bigl( v^*\left(\lambdacan+\pi^*\theta_t\right)-H_t(v_t)\Bigr)\, dt
   -F_0(v_0)
   \\
   &\quad
   =\int_0^1\left( (p_t+\theta_t|_{q_t})\dot q_t
   -\tfrac12\Abs{p_t}^2-\phi_t(q_t)\right) dt
   -f_0(q_0)
   .
\end{split}
\end{equation*}
\end{definition}

Again, see Definition~\ref{def:class-action}
and Remark~\ref{rem:twist-term},
the integral is over the interval $[0,1]$, and not the
circle $\SS^1$.
The subtraction of the \textbf{twist term}
$F_0(v_0)=f_0(q_0)$ is important to get
as critical points periodic solutions of the
$(\AAA,\phi)$-equation.
If $\theta_{t+1}=\theta_t$ is periodic, then the twist function
$f_t\equiv 0$ vanishes at every time $t$.

\subsection{Critical points}\label{sec:crit-Aa-4}

\begin{proposition}\label{prop:crit-Aa_H}
A loop $v=(q,\ppp)$ is a critical point of the symplectic action
$\Aa_{\lambda,H}\colon\Lambda T^*\Qfrak\to\R$
iff $q$ is a periodic solution of
$(\AAA,\phi)$-equation~(\ref{eq:q})
and~$\ppp=\dot q$.
\end{proposition}

\begin{proof}
In this proof, in order to avoid repetition, we already use parts of
Section~\ref{sec:Euler-VF} below on general symplectic manifolds.
By Theorem~\ref{thm:Euler}
the critical points of $\Aa_{\lambda,H}$ are
smooth solutions $v=(q,\ppp)\colon\SS^1\to \Qfrak\times \R^2$
of the Euler-Hamilton or
\textbf{\boldmath$(\lambda,H)$-equation} $\dot v=X(v)+Y(v)$,
see~(\ref{eq:Euler-Ham}).

\smallskip\noindent
'$\Rightarrow$'
We must calculate the Hamiltonian vector field $X$ and the Euler
vector field~$Y$ determined by~(\ref{eq:X-Y}).
For the  Euler field~$Y$ we need the time derivative
$$
   \dot\lambda_t=\pi^*\dot\theta_t
   =\dot A^1_t dq_1+\dot A^2_t dq_2
$$
pointwise along $T^*\Qfrak$ and at time $t\in\R$ and the symplectic form
\begin{equation*}
\begin{split}
   \omega_t
   &=\omegacan+\pi^* d\theta_t
\\
   &=dp_1\wedge dq_1+dp_2\wedge dq_2
   + \underbrace{\left(-\p_{q_2}A^1_t +\p_{q_1}A^2_t\right)}
   _{=\rot\,\AAA_t} dq_1\wedge dq_2 .
\end{split}
\end{equation*}
Then $Y_t$, implicitly defined by $\dot\lambda_t=\omega_t(\cdot, Y_t) $,
is explicitly given by the formula
$$
   Y_t=-\dot A_t^1\tfrac{\p}{\p p_1}-\dot A_t^2\tfrac{\p}{\p p_2}.
$$
This is a vertical vector field and it only depends
on points of the base $\Qfrak$ of~$T^*\Qfrak$.
To compute the Hamilton vector field $X$ observe that
$$
   dH_t
   = p_1\, dp_1+p_2\, dp_2+(\p_{q_1}\phi_t)\, dq_1+(\p_{q_2}\phi_t)\, dq_2.
$$
Then $X_t$, implicitly defined by $dH_t=\omega_t(\cdot, X_t) $,
is explicitly given by the formula
$$
   X_t
   =p_1 \tfrac{\p}{\p q_1}
   +p_2 \tfrac{\p}{\p q_2}
   +\left(p_2\, \rot\,\AAA_t - \tfrac{\p \phi_t}{\p q_1}\right) \tfrac{\p}{\p p_1}
   -\left(p_1\, \rot\,\AAA_t  + \tfrac{\p \phi_t}{\p q_2}\right) \tfrac{\p}{\p
     p_2} .
$$
Summarizing
$$ %
   X_t+Y_t
   =\begin{pmatrix}p_1\\p_2\\
      p_2\, \rot\,\AAA_t - \frac{\p \phi_t}{\p q_1} \\
        -p_1\, \rot\,\AAA_t - \frac{\p \phi_t}{\p q_2}
   \end{pmatrix}
   +\begin{pmatrix}0\\0\\-\AAAdot_t^1\\-\AAAdot_t^2\end{pmatrix}.
$$
Therefore the $(\lambda,H)$-equation~(\ref{eq:Euler-Ham})
is the ODE~system
\begin{equation*}
  \left\{
   \begin{aligned}
     \begin{pmatrix}\dot q_1\\\dot q_2\end{pmatrix}
     &=\begin{pmatrix}p_1\\p_2\end{pmatrix}
\\
     \begin{pmatrix}\dot p_1\\\dot p_2\end{pmatrix}
     &=
     \begin{pmatrix}
        p_2\, \rot\,\AAA_t|_q - \frac{\p \phi_t}{\p q_1}|_q -\AAAdot_t^1|_q\\
        -p_1\, \rot\,\AAA_t|_q - \frac{\p \phi_t}{\p q_2}|_q-\AAAdot_t^2|_q
     \end{pmatrix}
   \end{aligned}
   \right.
\end{equation*}
for smooth loops $(q,\ppp)\colon\SS^1\to T^*\Qfrak$.
Equivalently, the $(\lambda,H)$-equation is
\begin{equation*}
  \left\{
   \begin{aligned}
     \dot q
     &=\ppp
\\
     \dot\ppp
     &=-\left(\rot\,\AAA_t|_q\right) J_0 \ppp 
     -\AAAdot_t|_q
     -\Nabla{}\phi_t|_q
   \end{aligned}
   \right. .
\end{equation*}
Write this $1^{\rm st}$ order ODE system
as a $2^{\rm nd}$ order system to get the $(\AAA,\phi)$-equation
\begin{equation*}
   \ddot q
   =
{\color{gray}
\underbrace{{\color{black}-\left(\rot\,\AAA_t|_q\right) J_0 \dot q}}
      _{\text{Lorentz force}}
   \underbrace{{\color{black}-\AAAdot_t|_q}}_{\text{Euler\,f.}}
   \underbrace{{\color{black}-\Nabla{} \phi_t(q)}}_{\text{conservative}}
}
\end{equation*}
for smooth loops $q\colon \SS^1\to \Qfrak$.
This completes the proof of '$\Rightarrow$'.

\smallskip\noindent
'$\Leftarrow$'
Let $q$ be a periodic solution of~(\ref{eq:q})
and define $\ppp:=\dot q$.
Then $\dot q=\ppp$ and $\dot\ppp=\ddot q$ satisfy the
$(\lambda,H)$-equation displayed above.
This proves Proposition~\ref{prop:crit-Aa_H}.
\end{proof}

\section{Euler vector field}\label{sec:Euler-VF}

\begin{definition}\label{def:tw-per}
Let $(M,d\lambda_t)_{t\in\R}$
be a smooth\footnote{
  here smoothness refers to
  $\R\times TM\times\ni (t,(x,\xi))\mapsto \lambda_t|_x \xi\in \R$
  being smooth
  }
family of exact symplectic manifolds
such that $\lambda_t$ is \textbf{twisted-periodic},
i.e. at each time $t$ it holds that
\begin{equation}\label{eq:tw-per}
   \dot\lambda_{t+1}=\dot\lambda_t
   ,\qquad
   \lambda_{t+1}=\lambda_t+dF_t
   ,\qquad
{\color{gray}
   \omega_t:=d\lambda_t=d\lambda_{t+1}
   ,
}
\end{equation}
for a smooth function $F_t\colon M\to \R$,
called a \textbf{twist function}.
The time-constant function $F:=F_0\colon M\to\R$
is a twist function.\footnote{
  Indeed
  $
   \lambda_{t+1}-\lambda_t-dF
   =\underline{\lambda_{t+1}-\lambda_t-dF_t}
   +d(F_t-F_0)
   =d\int_0^t\dot F_s\, ds
   =0
  $ since $d \dot F_s=0$.
  }
\end{definition}

Taking the time derivative of identity two in~(\ref{eq:tw-per})
tells $d\dot F_t=\dot\lambda_{t+1}-\dot\lambda_t=0$.
Thus the time derivative $\dot F_t(x)$ of a general twist function
$F_t$ is locally constant in $x$, not necessarily in $t$.

\smallskip
Observe that the symplectic form depends periodically on time
$\omega_{t+1}=\omega_t$.
Let $H\colon \SS^1\times M\to\R$ be a smooth function on $M$ depending
periodically on time. At each time $t$
non-degeneracy of the symplectic form $\omega_t$
allows for transforming the $1$-forms $dH_t$ and $\dot\lambda_t$ into
vector fields along $M$, namely
\begin{equation}\label{eq:X-Y}
   dH_t=d\lambda_t(\cdot, X_t)
   ,\qquad
   \dot\lambda_t=d\lambda_t(\cdot, Y_t) .
\end{equation}
Note that both $X_t$ and $Y_t$ are periodic.
While $X_t=X_{H_t}^{d\lambda_t}$ is called \textbf{Hamiltonian vector field},
we refer to 
$$
   Y_t=Y_{\dot\lambda_t}^{d\lambda_t}
$$
as \textbf{Euler vector field}, because of its relation to the Euler
force as explained in the merry-go-round example in
Section~\ref{sec:merry}.

\subsection{Hamiltonian action functional and critical points}
\label{sec:crit-Aa}

\begin{definition}\label{def:symplectic-action}
For $(M,d\lambda_t)_{t\in\R}$ twisted-periodic
define the \textbf{action functional}
\begin{equation*}
\begin{split}
   \Aa=\Aa_{\lambda,H}\colon\Lambda M:=W^{1,2}(\SS^1,M)
   &\to\R
\\
   v
   &\mapsto
   \int_0^1 \left(v^*\lambda-H(t,v_t)\right) dt
   -F_0(v_0)
   .
\end{split}
\end{equation*}
\end{definition}

Again, as in Definition~\ref{def:class-action},
the integral is over the interval $[0,1]$ and not the
circle $\SS^1$.
The subtraction of the \textbf{twist term}
$F_0(v_0)$, see~(\ref{eq:tw-per}),
is important to get
as critical points periodic solutions of the
$(\lambda,H)$-equation~(\ref{eq:Euler-Ham}).
If $\lambda_{t+1}=\lambda_t$ is periodic, then twist function
and twist term vanish.

\begin{theorem}[Euler-Hamilton equation]\label{thm:Euler}
Critical points of $\Aa=\Aa_{\lambda,H}$ are
smooth solutions $v\colon\SS^1\to M$ of the \textbf{Euler-Hamilton}
or \textbf{\boldmath$(\lambda,H)$-equation}
\begin{equation}\label{eq:Euler-Ham}
   \dot v
   =X_t(v)+Y_t(v)
\end{equation}
where $X_t=X_{H_t}^{d\lambda_t}$ and $Y_t=Y_{\dot\lambda_t}^{d\lambda_t}$
are determined by~(\ref{eq:X-Y}).
\end{theorem}

\begin{definition}[2-form $\Omega:=D\Lambda$ on $\SS^1\times M$]
\label{def:Omega}
Let us move to $\R\times M$.
Let $D$ be the exterior derivative on $\R\times M$.
The family of $1$-forms $\lambda_t$ on $M$
induces a $1$-form $\Lambda$ on $\R\times M$, namely
$$
   \Lambda|_{(t,x)}
   \begin{pmatrix}r\\\xi\end{pmatrix}
   :=\lambda_t|_x\xi
$$
whenever $(t,x)\in \R\times M$ and $(r,\xi)\in \R\times T_xM$.
The exterior derivative is
\begin{equation*}
\begin{split}
    \Omega
   :=D\Lambda|_{(t,x)}
   \left(\begin{pmatrix}r\\\xi\end{pmatrix},
   \begin{pmatrix}s\\\eta\end{pmatrix}\right)
   &=d\lambda_t|_x (\xi,\eta)
   +\left(dt\wedge\dot\lambda_t|_x\right)
   \left(\begin{pmatrix}r\\\xi\end{pmatrix},
   \begin{pmatrix}s\\\eta\end{pmatrix}\right)
\\
   &=\omega_t|_x (\xi,\eta)
   +\dot\lambda_t|_x\eta\cdot r
   -\dot\lambda_t|_x\xi\cdot s .
\end{split}
\end{equation*}
By twisted periodicity~(\ref{eq:tw-per}) this is actually a $2$-form
on $\SS^1\times M$.
\end{definition}

The following proposition gives a different description of the
non-Hamiltonian part of the functional which is useful to compute the
critical points.

\begin{proposition}[Non-Hamiltonian part -- alternative description]
\label{prop:integration-HOmega}
Let $\lambda$ and $\Omega$ be as in Definition~\ref{def:Omega}.
Fix a base loop $b\in\Lambda M$.
Then every loop $v$ homotopic to $b$ in $\Lambda M$
through a smooth homotopy, say $\mathfrak{h}$ from $b$ to $q$,
satisfies the identity
$$
   \int_{[0,1]} v^*\lambda
   -F_0(v_0)
   =\int_R \mathfrak{H}^*\Omega
   +\int_{[0,1]} b^*\lambda
   -F_0(b_0)
$$
with $\mathfrak{H}\colon R=[0,1]\times[0,1]\to \SS^1\times M$,
$(r,t)\mapsto\left(t,\mathfrak{h}(r,t)\right)$ and $F_0$
from~(\ref{eq:tw-per}).
\end{proposition}

\begin{proof}
Literally the same proof as Proposition~\ref{prop:integration-HTheta}.
\end{proof}

\begin{proof}[Proof of Theorem~\ref{thm:Euler}]
Let
$\mathfrak{h} \colon Z=[0,1]\times \SS^1\to M $,
$(r,t)\mapsto \mathfrak{h} (r,t)$, be a smooth homotopy of loops from
$b=\mathfrak{h} (0,\cdot)$ to $v=\mathfrak{h} (1,\cdot)$.

\noindent
To construct variations with real parameter $\tau$
pick a Riemannian metric on~$M$ and let $\exp$ be the associated
exponential map.
We vary $v$ by $v_{(\tau)}:=\exp_v \tau\xi$ and
we add to the homotopy $\mathfrak{h} $ from $b$ to $v$
the linear homotopy from $v$ to $v_{(\tau)}$ extending the parameter
interval to $r\in[0,1+\tau]$, in symbols
$$
   \mathfrak{h} _{(\tau)}(r,\cdot)
   :=
   \begin{cases}
      \mathfrak{h} (r,\cdot)&\text{, $r\in[0,1]$,}\\
      \exp_v (r-1)\xi &\text{, $r\in[1,1+\tau]$.}
   \end{cases}
$$
We lift $\mathfrak{h} _{(\tau)}$ to $\mathfrak{H} _{(\tau)}(r,t)
:=(t,\mathfrak{h} _{(\tau)}(r,t))$.
Set $\omega_t:=d\lambda_t$.
Using the formula for $\Omega=D\Lambda$
in Definition~\ref{def:Omega}, one computes in complete analogy to the
lengthy calculation after~(\ref{eq:h}) the underlined
\underline{identity} shown in step~4 below.
In step~3 below we apply
Proposition~\ref{prop:integration-HOmega}
for $v_{(\tau)}:=\exp_v \tau\xi$ and $\mathfrak{h}_{(\tau)}$
to obtain
\begin{equation*}
\begin{split}
   &d \Aa|_{v}\xi
\\
   &=\left.\tfrac{d}{d\tau}\right|_{\tau=0} \Aa(\exp_v \tau\xi)
\\
   &=-\int_0^1 \left.\tfrac{d}{d\tau}\right|_{0}
   H_t(\exp_{v_t} \tau\xi_t) \,
   dt
   +
   \left.\tfrac{d}{d\tau}\right|_{0}
   \biggl(
   \int_{[0,1]} v_{(\tau)}^*\lambda
   -F_0(v_{(\tau)}(0))
   \biggr)
\\
   &\stackrel{{\color{gray}3}}{=}
   -\int_0^1 dH_t|_{v_t} \xi_t\, dt
   +
\underline{
   \left.\tfrac{d}{d\tau}\right|_{0}
   \biggl(
   \int_R \mathfrak{H}_{(\tau)}^*\Omega+\int_{[0,1]} b^*\lambda-F_0(b_0)
   \biggr)
}
\\
\end{split}
\end{equation*}
\begin{equation*}
\begin{split}
   &\stackrel{{\color{gray}4}}{=}
   -\int_0^1\omega_t|_{v_t}\bigl(\xi_t,X_t|_{v_t}\bigr) \, dt
   +
\underline{
  \int_0^1
   \left(
   \omega_t|_{v_t}\bigl(\xi_t,\dot v_t\bigr)
   -\dot\lambda_t|_{v_t}\xi_t
   \right)
   dt
}
\\
   &\stackrel{{\color{gray}5}}{=}
   \int_0^1\omega_t|_{v_t}\bigl(X_t|_{v_t}+Y_t|_{v_t}-\dot v_t,\xi_t\bigr)\, dt
   .
\end{split}
\end{equation*}
Equalities~4 and~5 use that
$X_t=X_{H_t}^{\omega_t}$ and $Y_t=Y_{\dot\lambda_t}^{\omega_t}$
are determined by~(\ref{eq:X-Y}).

Now suppose that $v$ is a critical point of $\Aa$, i.e. $d \Aa|_{v}=0$.
Then the identity we just proved tells that
$\dot v=X_t(v)+Y_t(v)$.
Note that the right hand side is in $W^{1,2}$, so $v\in W^{2,2}(\SS^1,M)$.
But then the right hand side is in $W^{2,2}$, so $v\in
W^{3,2}$ etc.
Hence $v\in \cap_{\ell\in\N_0 }W^{\ell,2}(\SS^1,M)=C^\infty(\SS^1,M)$.
This proves Theorem~\ref{thm:Euler}.
\end{proof}

\subsection{Eliminating the Hamiltonian vector field}
\label{sec:elim-Ham}

\begin{proposition}\label{prop:eliminate-Ham-VF}
Let $(M,d\lambda_t)_{t\in\R}$ be a twisted-periodic exact symplectic
manifold and $H\colon\SS^1\times M\to\R$ a smooth function.
Then the $1$-form family defined~by
$$
   \lambda^H_t
   :=\lambda_t+\int_0^t dH_s\, ds
   ,\qquad
   t\in\R
   ,
$$
is twisted-periodic. The critical points of the
functions $\Aa_{\lambda,H}$ and $\Aa_{\lambda^H,0}$ coincide.
\end{proposition}

\begin{proof}
We verify the three conditions in~(\ref{eq:tw-per}):
Condition one: We get that
$$
   \underline{d\lambda^H_t}
   =d\lambda_t+\int_0^1 ddH_s\, ds
   \:\underline{=d\lambda_t}
   \stackrel{\text{(\ref{eq:tw-per}))}}{=}d\lambda_{t+1}
   =d\lambda^H_{t+1}
$$
since $dd=0$ and the final identity
uses the underlined one at time $t+1$.
To show condition two use the definition of $\lambda_t^H$ to
get the first and last identity in
$$
   \dot\lambda^H_t
   =\dot\lambda_t+dH_t
   =\dot\lambda_{t+1}+dH_{t+1}
   =\dot\lambda^H_{t+1}
$$
where identity two uses periodicity of $\dot\lambda_t$ and of $H_t$.
Condition three: Note that
$$
   \lambda^H_{t+1}-\lambda^H_t
   \stackrel{{\color{gray}1}}{=}\lambda_{t+1}-\lambda_t+\int_t^{t+1} dH_s\, ds
   \stackrel{{\color{gray}2}}{=} d(F_t+\widebar H)
$$
where equality 1 is by definition of $\lambda^H_{t+1}$ and $\lambda^H_t$.
Equality 2 uses the last hypothesis in~(\ref{eq:tw-per}) and the
definition of the pointwise time-mean $\widebar H:=\int_t^{t+1} H_s\, ds$.
This proves that the family $\lambda^H_t$ is twisted-periodic.

To show that the critical points coincide it suffices to show 
that the $(\lambda,H)$- and the $(\lambda^H,0)$-equation are equal,
equivalently that the vector field identity
$$
   Y_t^H=Y_t+X_t
   ,\qquad
\text{\color{gray}\small$
   Y_t^H:=Y_{\dot\lambda^H_t}^{d\lambda^H_t}
   ,\quad
   Y_t:=Y_{\dot\lambda_t}^{d\lambda_t}
   ,\quad
   X_t:=X_{H_t}^{d\lambda_t} ,
$}
$$
holds true. By definition~(\ref{eq:X-Y}) of the respective
vector fields we get $1,3$ in
$$
   d\lambda^H_t(\cdot,Y_t^H)
   \stackrel{{\color{gray}1}}{=}\dot\lambda_t^H
   \stackrel{{\color{gray}2}}{=}\dot\lambda_t+dH_t
   \stackrel{{\color{gray}3}}{=}d\lambda_t(\cdot,Y_t)+d\lambda_t(\cdot,X_t)
   =d\lambda_t(\cdot,Y_t+X_t) .
$$
Here equality 2 holds by definition of $\lambda_t^H$ and the
fundamental theorem of calculus.
But $d\lambda^H_t=d\lambda_t$ since $dd=0$.
Now the equality $Y_t^H=Y_t+X_t$ follows by non-degeneracy
of the symplectic form $d\lambda_t$.
This proves Proposition~\ref{prop:eliminate-Ham-VF}.
\end{proof}

\subsection{Periodizing the twisted-periodic
symplectic primitive}\label{sec:periodizing}

\begin{proposition}\label{prop:periodizing}
Let $(M,d\lambda_t)_{t\in\R}$ be a twisted-periodic exact symplectic
manifold~(\ref{eq:tw-per})
and $H\colon\SS^1\times M\to\R$ smooth.
Then the two families defined~by
$$
   \tilde \lambda_t
   :=\lambda_t-t\, dF
   ,\qquad
   \tilde H_t
   :=H_t+F
   ,\qquad
{\color{gray}
   F:=F_0
   ,
}
$$
are periodic. Furthermore, the critical points of
$\Aa_{\lambda,H}$ and $\Aa_{\tilde\lambda,\tilde H}$ coincide.
\end{proposition}

\begin{proof}
Since $H_t$ is periodic in $t$, the family
$\tilde H_t$ is obviously periodic in $t$.
To check that $\tilde\lambda_t$ is periodic in $t$, we compute
\begin{equation*}
\begin{split}
   \tilde\lambda_{t+1}-\tilde\lambda_t
   =\lambda_{t+1}-(t+1) dF-\lambda_t+t\, dF
   =\lambda_{t+1}-\lambda_t-dF
   =0
   .
\end{split}
\end{equation*}
By~(\ref{eq:X-Y}) the Euler vector fields are determined by
$\dot\lambda_t=d\lambda_t(\cdot,Y_t)$ and
$\dot{\tilde\lambda}_t=d\tilde\lambda_t(\cdot,\tilde Y_t)$.
Using that $d\tilde\lambda_t=d\lambda_t$, we compute
$$
   d\lambda_t(\cdot,Y_t-\tilde Y_t)
   =\dot\lambda_t-\dot{\tilde\lambda}_t
   =\dot\lambda_t-(\dot\lambda_t-dF)
   =dF
   =d\lambda_t(\cdot,X_F^{d\lambda_t})
   .
$$
Since $d\lambda_t$ is symplectic, hence non-degenerate,
we obtain that $\tilde Y_t=Y_t-X_F^{d\lambda_t}$.
By~(\ref{eq:X-Y}), now for the Hamilton vector fields, 
since $\tilde H_t-H_t=F$ we get
$$
   d\lambda_t(\cdot,\tilde X_t-X_t)
   =d\tilde H_t-dH_t
   =dF
   =d\lambda_t(\cdot,X_F^{d\lambda_t})
   ,
$$
thus $\tilde X_t=X_t+X_F^{d\lambda_t}$.
Consequently $\tilde X_t+\tilde Y_t=X_t+Y_t$ which means that
the $(\tilde \lambda,\tilde H)$- and the $(\lambda,H)$-equation
are equal and, by Theorem~\ref{thm:Euler},
so are the critical points of $\Aa_{\tilde \lambda,\tilde H}$ and
$\Aa_{\lambda,H}$.
This proves Proposition~\ref{prop:periodizing}.
\end{proof}

\subsection{The Euler flow is symplectic}

Let $(M,\lambda_t)_{t\in\R}$ be a twisted-periodic exact symplectic
manifold.
To simplify notation we assume in the following
that the flow $\varphi^t_Y$ of the Euler vector field $Y_t$ globally exists,
i.e. for any $t\in\R$ there is a diffeomorphism $\varphi^t_Y\colon M\to M$
such that $\varphi^0_Y=\id_M$ and
$\frac{d}{dt}\varphi^t_Y=Y_t\circ \varphi^t_Y$.

\begin{proposition}\label{prop:Euler-flow-sympl}
The Euler flow is symplectic, in symbols
$(\varphi^t_Y)^*\omega_0=\omega_t$.
\end{proposition}

\begin{proof}
Since the inverse of $\varphi^{t}_Y$ is $\varphi^{-t}_Y$, the identity 
$$
   \omega_t=(\varphi^t_Y)^*\omega_0
{\color{gray}\;
   =((\varphi^{-t}_Y)^{-1})^*\omega_0
   =((\varphi^{-t}_Y)^{*})^{-1}\omega_0
}
$$
is equivalent to the identity $(\varphi^{-t}_Y)^*\omega_t=\omega_0$.
To prove this identity it suffices to show that the function $t\mapsto
(\varphi^{-t}_Y)^*\omega_t$ is constant since at time zero
the value is $(\varphi^{0}_Y)^*\omega_0=\id_M^*\omega_0=\omega_0$.
Indeed by Cartan's formula
$$
   L_{Y_t}\omega_t
   =di_{Y_t}\omega_t+i_{Y_t}d \omega_t
   =-d\dot\lambda_t
   =-\dot\omega_t
$$
and then by the Leibniz rule the derivative vanishes
$$
   \tfrac{d}{dt} (\varphi^{-t}_Y)^*\omega_t
   =(\varphi^{-t}_Y)^*\bigl(-L_{Y_t}\omega_t\bigr)
   +(\varphi^{-t}_Y)^*\dot\omega_t
   =0 .
$$
This proves Proposition~\ref{prop:Euler-flow-sympl}.
\end{proof}

\subsection{Applying Cartan's formula on the loop space}
\label{sec:Cartan}

In Section~\ref{sec:Cartan} we restrict to the \emph{periodic}
case ($\theta_{t+1}=\theta_t$).
We show how Theorem~\ref{thm:Euler} is
related to application of Cartan's formula
on the loop space.
\\
Since the loop space is infinite dimensional,
to the best of our knowledge,
the Cartan formula is not established in this setup.
However, since the formal application of Cartan's formula
coincides with the result of Theorem~\ref{thm:Euler},
this gives evidence that Cartan's formula is valid on the loop space
as well.

\smallskip
We first consider the following geometric setup.
Assume that $N$ is a manifold, $\Lambda\in\Omega^1(N)$
is a $1$-form on $N$, and $\Vv\in\Gamma(TN)$ is a vector field on $N$.
If $N$ is finite dimensional the following discussion is completely
rigorous. However, we want to apply the discussion below to the case
where $N$ is the loop space of a finite dimensional manifold.
Plugging in the vector field into the $1$-form we obtain a function
$$
   f:=i_\Vv \Lambda\colon N\to\R .
$$
By Cartan's formula the differential is given by
$$
   df
   {\color{gray}\,=di_\Vv \Lambda\,}
   =L_\Vv \Lambda-i_\Vv  d\Lambda .
$$
In the special case where $\Omega:=d\Lambda$ is symplectic
we can define a further vector field $\Yy$ on $N$,
called \textbf{Euler vector field}, by the requirement of equal $1$-forms
$$
   i_\Yy\Omega=L_\Vv \Lambda .
$$
In this case the differential can be written as 
$
   df=i_{\Yy-\Vv }\Omega
$
and therefore critical points of $f$ are points $z\in N$ satisfying
the \textbf{abstract Euler equation}
\begin{equation}\label{eq:abstract-Euler}
   \Vv (z)=\Yy(z).
\end{equation}

\boldmath
\subsubsection*{Example: Loop space}
\unboldmath

In the following we apply this observation to the loop space case
$N=\Ll(M):=C^\infty(\SS^1,M)$ with its \textbf{canonical vector field}
and a $1$-form on $\Ll M$, namely
$$
   \Vv (z)=\dot z:=\p_t z
   ,\qquad
   \Lambda:=\int_0^1\lambda_t\, dt ,
$$
where $\{\lambda_t\}_{t\in\SS^1}$ is a periodic $1$-form
on $M$ such that the periodic family $\omega_t:=d\lambda_t$ consists
of symplectic forms on $M$.
The flow of $\Vv =\p_t$ on $\Ll M$ is
$
   \Phi^r_\Vv  z=r_*z
$
where $(r_*z)(t)=z(t+r)$ for every time $t$.
Linearization yields
\begin{equation*}
   (\Phi^r_\Vv  z)_t=z_{t+r}
   ,\qquad
   \left(d \Phi^r_\Vv|_z\xi\right)_t=\xi_{t+r}
   .
\end{equation*}
Let $z\in\Ll M$ and $\xi\in T_z\Ll M$. We compute the pull-back
\begin{equation*}
\begin{split}
   \left({\Phi^r_\Vv }^*\Lambda\right)_z\xi
   :=\Lambda_{\Phi^r_\Vv (z)}\, d\Phi^r_\Vv |_z\xi
   &=
   \int_0^1 \lambda_t|_{z_{t+r}}
   \xi_{t+r}\, dt
\\
   &\stackrel{{\color{gray}2}}{=}
   \int_{r}^{r+1}\lambda_{t-r}|_{z_t} \xi_t\, dt
\\
   &\stackrel{{\color{gray}3}}{=}
   \int_0^1\lambda_{t-r}|_{z_t} \xi_t\, dt
    .
\end{split}
\end{equation*}
Equality 2 is by change of variables.
Equality 3 is by periodicity $\lambda_{t+1}=\lambda_t$.
The Lie derivative of $\Lambda$ with respect to $\Vv $
is by definition
$$
   L_\Vv \Lambda
   :=\left.\tfrac{d}{dr}\right|_{r=0}{\Phi^r_\Vv}^*\Lambda
   =-\int_0^1 \dot\lambda_t\, dt .
$$
The exterior derivative of $\Lambda$ is symplectic, namely
$$
   \Omega
   :=d\Lambda
   =\int_0^1 \omega_t\, dt .
$$
Therefore the Euler vector field localizes in the sense that
$$
   \omega_t|_{z_t}\left(\Yy|_z(t),\cdot\right)=-\dot\lambda_t|_{z_t} .
$$
Hence the loop space Euler vector field
$$
   \Yy|_z(t)
   =Y_t|_{z_t}
$$
coincides with the Euler vector field
$Y$ along $M$ as defined by~(\ref{eq:X-Y}).

In particular, in this example the abstract Euler equation~(\ref{eq:abstract-Euler})
gives rise to the manifold Euler equation
$$
   \dot z_t=Y_t|_{z_t}
$$
as obtained earlier,
see~(\ref{eq:Euler-Ham}) with vanishing $H$, hence vanishing $X$.

\appendix






\bibliographystyle{alpha}
\addcontentsline{toc}{section}{References}
\small
\bibliography{$HOME/Dropbox/0-Libraries+app-data/Bibdesk-BibFiles/library_math,$HOME/Dropbox/0-Libraries+app-data/Bibdesk-BibFiles/library_math_2020,$HOME/Dropbox/0-Libraries+app-data/Bibdesk-BibFiles/library_physics}{}

%


\end{document}